\documentclass{amsart}
\usepackage{amssymb, latexsym}
\usepackage{url}

\DeclareMathOperator{\lcm}{\mathrm{lcm}}

\begin{document}
 \bibliographystyle{plain}

 \newtheorem{theorem}{Theorem}
 \newtheorem{lemma}{Lemma}
 \newtheorem{corollary}{Corollary}
 \newtheorem{conjecture}{Conjecture}
 \newtheorem{definition}{Definition}
 \newtheorem{proposition}{Proposition}
 \newtheorem*{main}{Main~Theorem}
 \newcommand{\mc}{\mathcal}
 \newcommand{\rar}{\rightarrow}
 \newcommand{\Rar}{\Rightarrow}
 \newcommand{\lar}{\leftarrow}
 \newcommand{\lrar}{\leftrightarrow}
 \newcommand{\Lrar}{\Leftrightarrow}
 \newcommand{\zpz}{\mathbb{Z}/p\mathbb{Z}}
 \newcommand{\mbb}{\mathbb}
 \newcommand{\A}{\mc{A}}
 \newcommand{\B}{\mc{B}}
 \newcommand{\cc}{\mc{C}}
 \newcommand{\D}{\mc{D}}
 \newcommand{\E}{\mc{E}}
 \newcommand{\F}{\mc{F}}
 \newcommand{\G}{\mc{G}}
  \newcommand{\ZG}{\Z (G)}
 \newcommand{\FN}{\F_n}
 \newcommand{\I}{\mc{I}}
 \newcommand{\J}{\mc{J}}
 \newcommand{\M}{\mc{M}}
 \newcommand{\nn}{\mc{N}}
 \newcommand{\qq}{\mc{Q}}
 \newcommand{\U}{\mc{U}}
 \newcommand{\X}{\mc{X}}
 \newcommand{\Y}{\mc{Y}}
 \newcommand{\itQ}{\mc{Q}}
 \newcommand{\C}{\mathbb{C}}
 \newcommand{\R}{\mathbb{R}}
 \newcommand{\N}{\mathbb{N}}
 \newcommand{\Q}{\mathbb{Q}}
 \newcommand{\Z}{\mathbb{Z}}
 \newcommand{\ff}{\mathfrak F}
 \newcommand{\fb}{f_{\beta}}
 \newcommand{\fg}{f_{\gamma}}
 \newcommand{\gb}{g_{\beta}}
 \newcommand{\vphi}{\varphi}
 \newcommand{\whXq}{\widehat{X}_q(0)}
 \newcommand{\Xnn}{g_{n,N}}
 \newcommand{\lf}{\left\lfloor}
 \newcommand{\rf}{\right\rfloor}
 \newcommand{\lQx}{L_Q(x)}
 \newcommand{\lQQ}{\frac{\lQx}{Q}}
 \newcommand{\rQx}{R_Q(x)}
 \newcommand{\rQQ}{\frac{\rQx}{Q}}
 \newcommand{\elQ}{\ell_Q(\alpha )}
 \newcommand{\oa}{\overline{a}}
 \newcommand{\oI}{\overline{I}}
 \newcommand{\dx}{\text{\rm d}x}
 \newcommand{\dy}{\text{\rm d}y}
 \newcommand{\id}{\text{d}}

\title[The group ring of $\Q/\Z$]{The group ring of $\Q/\Z$ and an\\ application of a divisor problem}
\author{Alan K. Haynes and Kosuke Homma}
\thanks{Research of the first author supported by EPSRC grant EP/F027028/1}
\address{Department of Mathematics, University of York, Heslington, York YO10 5DD, UK}
\email{akh502@york.ac.uk}
\address{Department of Mathematics,
University of Texas, Austin, Texas 78712,
USA}\email{khomma@math.utexas.edu} \allowdisplaybreaks

%%%%%%%%%%%%%%%%%%%%%%%%%%%%%%%%%%%%%%%%%%%%%%%%%%%%%%%%%%%%%%%%%%%%%%%%%%%%%%%%
%%%%%%%%%%%%%%%%%%%%%%%%%%%%%%%%%%%%%%%%%%%%%%%%%%%%%%%%%%%%%%%%%%%%%%%%%%%%%%%%

\begin{abstract}
First we prove some elementary but useful identities in the group
ring of $\Q/\Z$. Our identities have potential applications to
several unsolved problems which involve sums of Farey fractions.
In this paper we use these identities, together with some
analytic number theory and results about divisors in short
intervals, to estimate the cardinality of a class of sets of
fundamental interest.
\end{abstract}

%%%%%%%%%%%%%%%%%%%%%%%%%%%%%%%%%%%%%%%%%%%%%%%%%%%%%%%%%%%%%%%%%%%%%%%%%%%%%%%%
%%%%%%%%%%%%%%%%%%%%%%%%%%%%%%%%%%%%%%%%%%%%%%%%%%%%%%%%%%%%%%%%%%%%%%%%%%%%%%%%

\maketitle
\section{Introduction}
Let $G$ be the multiplicative group defined by
\[G=\{z^\beta:\beta\in\Q/\Z\},\]
and let $(\ZG,+,\times)$ denote the group ring of $G$ with
coefficients in $\Z$. Clearly $G$ is just $\Q/\Z$ written
multiplicatively, but depending on the context it may be more
appropriate to speak either of addition in $\Q/\Z$ or of
multiplication in $G$. The group ring $\ZG$ is just the ring of
formal polynomials with integer coefficients, evaluated at
elements of $G$.

Questions about $\ZG$ arise naturally in several elementary
problems in number theory. Two interesting examples are the
estimation of powers of exponential sums and the problem of
approximation of real numbers by sums of rationals. Furthermore
the study of the additive structure of arithmetically interesting
groups is a fruitful area which has led to significant recent
developments \cite{tao2006}. The purpose of this paper is to
establish some basic algebraic results about $\ZG$ and then to
show how recent techniques for dealing with divisors in intervals
(\cite{cobeli2003},\cite{ford2008},\cite{ford2006}) can be
applied to estimate the cardinality of an arithmetically
interesting class of subsets of $G$.

For each $\beta\in\Q/\Z$ we denote the additive order of $\beta$
by $h(\beta)$. Then for each positive integer $Q$ we define a
subset $\F_Q$ of $\Q/\Z$ by
\[\F_Q=\{\beta\in\Q/\Z:h(\beta)\le Q\}.\]
Clearly the set $\F_Q$ can be identified with the Farey fractions
of order $Q$. For $Q\ge 3$ the set
\[\left\{z^{\beta}\in G:\beta\in\F_Q\right\}\subset G\]
is not closed under multiplication. However it does generate a
finite subgroup of $G$, which we call
\[G_Q=\left<z^{\beta}:\beta\in\F_Q\right>.\]
Obviously this definition of $G_Q$ is also valid if $Q<3$. It is easy to check that
\begin{equation}\label{GQdef1}
G_Q=\left\{z^\beta:\beta\in\Q/\Z,~ h(\beta)\mid\lcm\{1,2,\ldots ,Q\}\right\}.
\end{equation}
The cardinality of $G_Q$ is given by
\begin{align*}
|G_Q|=\sum_{q|\lcm\{1,2,\ldots ,Q\}}\varphi (q)=\lcm\{1,2,\ldots
,Q\}=\exp\left(\sum_{q\le Q}\Lambda (q)\right),
\end{align*}
where $\Lambda$ denotes the von-Mangoldt function. Thus by the
Prime Number Theorem there is a positive constant $c_1$ for which
\begin{align}\label{GQorder}
|G_Q|=\exp\left(Q+O\left(Qe^{-c_1\sqrt{\log Q}}\right)\right).
\end{align}
It appears to be a somewhat more difficult problem to
estimate the cardinality of the set
\begin{align}\label{FQsumset1}
\overbrace{\F_Q+\cdots +\F_Q}^{k-\text{times}}
\end{align}
when $k\ge 2$ is a small positive integer. The main result of our paper is the following theorem.
\begin{theorem}\label{FQordthm}
With $\delta =1- \frac{1+\log \log 2}{\log 2}$ we have as $Q
\rightarrow \infty$ that
\begin{align*}
|\F_Q+\F_Q|\asymp \frac{Q^4}{(\log Q)^{\delta} (\log \log Q)^{3/2}}.
\end{align*}
\end{theorem}
Theorem \ref{FQordthm} will be proved in Section 3. The most
technical part of the proof uses results of K. Ford
\cite{ford2006} about the distribution of integers whose divisors
have certain properties. Results of this type were also used in
\cite{cobeli2003} to study gaps between consecutive Farey
fractions. Before we get to our proof of Theorem \ref{FQordthm}
we will develop some important tools that give us information
about the corresponding elements of $\ZG$. For each nonnegative
integer $q$ we define $F_q\in\ZG$ by
\[F_q(z)=\sum_{\substack{\beta\in\Q/\Z \\ h(\beta)=q}}z^\beta=\sum_{\substack{a=1 \\ (a,q)=1}}^qz^{a/q}.\]
For conciseness, from on we will suppress the dependance of $F_q$
upon $z$. In Section \ref{secgrpring} we prove the following
general result.
\begin{theorem}\label{ZGaddthm}
Suppose that $q$ and $r$ are positive integers. Let $d=(q,r)$ and
let $d'$ be the largest divisor of $d$ which is relatively prime
to both $q/d$ and $r/d$. Then we have that
\begin{align}\label{ZGaddeqn1}
F_q\times F_r=\varphi (d)\sum_{e|d'}c(d',e)F_{qr/de},
\end{align}
where
\begin{equation*}
c(d',e)=\prod_{\substack{p|d'
\\p\nmid e}}\left(1-\frac{1}{p-1}\right).
\end{equation*}
\end{theorem}
There are several useful consequence of Theorem \ref{ZGaddthm},
some of which are formulated in the following two corollaries.
\begin{corollary}\label{ZGaddcor1}
With $q, r, d,$ and $d'$ as in Theorem \ref{ZGaddthm} we have
that
\begin{enumerate}
\item[(i)]If $d'=1$ then
\begin{equation*}
F_q\times F_r=\varphi (d)F_{qr/d},\quad\text{ and}
\end{equation*}
\item[(ii)] If $q$ and $r$ are squarefree then
\begin{align*}
F_q\times F_r=\varphi
(d)\sum_{e|d}\left(\prod_{p|e}\frac{p-2}{p-1}\right)F_{qre/d^2}.
\end{align*}
\end{enumerate}
\end{corollary}
\begin{corollary}\label{ZGaddcor2}
If $k$ is a positive integer then the set
\[\overbrace{\F_Q+\cdots +\F_Q}^{k-\text{times}}\]
consists of all elements $\beta\in\Q/\Z$ with $h(\beta)=
n_1n_2\cdots n_k$ for some positive integers $n_1,\ldots ,n_k\le
Q$ which satisfy $(n_i,n_j)=1$ for $i\not=j$.
\end{corollary}
As observed by the referee, we note that it is a simple matter to
prove Corollary \ref{ZGaddcor2} directly and that it is actually
all that we need for our proof of Theorem \ref{FQordthm}. However
there are other applications where the more general Theorem
\ref{ZGaddthm} is necessary. In particular, by using the large
sieve together with the group ring coefficients which appear in
our theorem we have been able to give a new proof
\cite{haynes2008} of an upper bound in metric number theory which
is crucial in the classical theory of the Duffin-Schaeffer
Conjecture \cite{harman1998}.

Finally we remark that it is not immediately clear how to extend
our results to estimate the number of elements in
(\ref{FQsumset1}) when $k>2$. We discuss this briefly at the end
of Section 3, where we also pose an open question of independent
interest.

Acknowledgements: We would like to thank our advisor Jeffrey
Vaaler for several discussions which influenced the direction of
our research on this problem, and for pointing out the formula
(\ref{GQorder}).

%%%%%%%%%%%%%%%%%%%%%%%%%%%%%%%%%%%%%%%%%%%%%%%%%%%%%%%%%%%%%%%%%%%%
%%%%%%%%%%%%%%%%%%%%%%%%%%%%%%%%%%%%%%%%%%%%%%%%%%%%%%%%%%%%%%%%%%%%

\section{The group ring of $\Q/\Z$}\label{secgrpring}
Our proof of Theorem \ref{ZGaddthm} depends on the following two
elementary lemmas.
\begin{lemma}\label{ZGaddlem1}
If $q$ and $r$ are relatively prime positive integers then
\begin{equation*}
F_q\times F_r=F_{qr}.
\end{equation*}
\end{lemma}
\begin{proof}
By our definitions we have that
\begin{equation}\label{ZGaddlem1eqn1}
F_q\times F_r=\sum_{\substack{a=1\\ (a,q)=1}}^q\sum_{\substack{b=1\\ (b,r)=1}}^rz^{(ar+bq)/qr}.
\end{equation}
If $c\in \Z,~(c,qr)=1$ then the equation
\[ar+bq=c\]
has a unique solution $(a,b)\in(\Z/q)^*\times (\Z/r)^*$. Conversely for any integers $a$ and $b$ with $(a,q)=(b,r)=1$ we have that $(ar+bq,qr)=1.$ Thus the map $(a,b)\mapsto ar+bq$ is a bijection from $(\Z/q)^*\times (\Z/r)^*$ onto $(\Z/qr)^*$. Comparing this with (\ref{ZGaddlem1eqn1}) now finishes the proof.
\end{proof}
\begin{lemma}\label{ZGaddlem2}
Let $p$ be prime and let $\alpha,\beta\in\Z, 1\le\alpha\le\beta$. Then we have that
\begin{align*}
F_{p^\alpha}\times F_{p^\beta}=
\begin{cases}
\varphi (p^\alpha)F_{p^\beta}&\text{ if }\alpha <\beta,\text{ and}\\
\varphi(p^\alpha)\sum_{i=0}^\alpha F_{p^i}-p^{\alpha-1}F_{p^\alpha}&\text{ if }\alpha=\beta.
\end{cases}
\end{align*}
\end{lemma}
\begin{proof}
First let us prove the case when $\alpha=\beta$. We have that
\begin{equation}\label{ZGaddlem2eqn1}
F_{p^\alpha}\times F_{p^\alpha}
=\sum_{\substack{a=1\\(a,p)=1}}^{p^\alpha}\sum_{\substack{b=1\\(b,p)=1}}^{p^\alpha}
z^{(a+b)/p^\alpha}.
\end{equation}
If $a$ is any integer with $(a,p)=1$ then working in $\Z/p^\alpha$ we have that
\begin{equation}\label{ZGaddlem2eqn2}
\left\{a+b~|~b\in(\Z/p^{\alpha})^*\right\}=\left(\Z/p^{\alpha}\right)\setminus\left\{a+np\mod p^\alpha~|~0\le n< p^{\alpha-1}\right\}.
\end{equation}
If $a'$ is another integer with $(a',p)=1$ and $a\not=a'\mod p$ then it is clear that
\begin{equation*}
a+np~\not=~a'+np\mod p^\alpha
\end{equation*}
for any $0\le n< p^{\alpha-1}$. Furthermore if $n$ and $n'$ are two integers with $0\le n<n'< p^{\alpha-1}$ then we also have that
\begin{equation*}
a+np~\not=~a+n'p\mod p^\alpha.
\end{equation*}
These comments reveal that for any integer $m$
\begin{equation*}
\left|\bigcup_{a=mp+1}^{(m+1)p-1}\left\{a+np\mod p^\alpha~|~0\le n< p^{\alpha-1}\right\}\right|=(p-1)p^{\alpha-1}.
\end{equation*}
Since $|(\Z/p^{\alpha})^*|=\varphi (p^\alpha)=(p-1)p^{\alpha-1}$ this implies that
\begin{equation*}
\bigcup_{a=mp+1}^{(m+1)p-1}\left\{a+np\mod p^\alpha~|~0\le n< p^{\alpha-1}\right\}=\left(\Z/p^{\alpha}\right)^*.
\end{equation*}
Combining this with (\ref{ZGaddlem2eqn1}) and (\ref{ZGaddlem2eqn2}) we find that
\begin{align*}
F_{p^\alpha}\times F_{p^\alpha}&=\varphi (p^{\alpha})\sum_{c=1}^{p^\alpha}z^{c/p^\alpha}-p^{\alpha-1}F_{p^{\alpha}}\\
&=\varphi (p^{\alpha})\sum_{i=0}^\alpha\sum_{\substack{c=1\\p^i\|c}}^{p^\alpha}z^{c/p^\alpha}-p^{\alpha-1}F_{p^{\alpha}}\\
&=\varphi (p^{\alpha})\sum_{i=0}^\alpha\sum_{\substack{d=1\\(d,p)=1}}^{p^{\alpha-i}}z^{d/p^{\alpha-i}}-p^{\alpha-1}F_{p^{\alpha}}\\
&=\varphi (p^{\alpha})\sum_{i=0}^\alpha F_{p^i}-p^{\alpha-1}F_{p^{\alpha}},
\end{align*}
and this is exactly what we were trying to show.

Now let us consider the easier case when $\alpha <\beta$. First observe that
\begin{equation}\label{ZGaddlem2eqn3}
F_{p^\alpha}\times F_{p^\beta}
=\sum_{\substack{a=1\\(a,p)=1}}^{p^\alpha}\sum_{\substack{b=1\\(b,p)=1}}^{p^\beta}
z^{(ap^{\beta-\alpha}+b)/p^\beta}.
\end{equation}
For each integer $a\in(\Z/p^{\alpha})^*$ we have that
\begin{align*}
\left\{ap^{\beta-\alpha}+b\mod p^{\beta}~|~b\in(\Z/p^{\beta})^*\right\}=\left(\Z/p^{\beta}\right)^*,
\end{align*}
and using this fact in (\ref{ZGaddlem2eqn3}) yields the desired result.
\end{proof}
Now we proceed to the proof of Theorem \ref{ZGaddthm}. The idea is
simply to split the factorizations of $q$ and $r$ into pieces
that we will then recombine using Lemmas \ref{ZGaddlem1} and
\ref{ZGaddlem2}.
\begin{proof}[Proof of Theorem \ref{ZGaddthm}]
Let $q_1, q_d, r_1,$ and $r_d$ be the unique positive integers
which satisfy
\begin{align*}
&q=q_1q_dd,~ r=r_1r_dd,\\
&(q_1,d)=(r_1,d)=1,~\text{ and}\\
&p~|~q_d\text{ or } r_d~\text{ and } p~\text{prime} \Rightarrow p~|~d,
\end{align*}
 and let $d_{qr}=d/d'$. It follows immediately that $(d',q_dr_d)=1$. Since $d'$ is the largest divisor of $d$ with $(q/d,d')=(r/d,d')=1$ it also follows that if $p$ is prime and $p~|~d_{qr}$ then $p~|~q_dr_d$. This implies that $(d_{qr},d')=1$. By Lemma \ref{ZGaddlem1} we have that
 \begin{align}
 F_q\times F_r&=(F_{q_1}\times F_{q_dd_{qr}}\times F_{d'})
 \times (F_{r_1}\times F_{r_dd_{qr}}\times F_{d'})\nonumber\\
 &=F_{q_1r_1}\times (F_{q_dd_{qr}}\times F_{r_dd_{qr}})\times (F_{d'}\times F_{d'}).\label{ZGaddthmeqn1}
 \end{align}
Now $(q/d,r/d)=1$ so we can find distinct primes $p_1,\ldots ,p_\ell, p_{\ell+1},\ldots ,p_k$ and positive integers $a_1,\cdots a_k$ for which
\begin{align*}
q_d&=p_1^{a_1}\cdots p_\ell^{a_\ell},\quad\text{ and}\\
r_d&=p_{\ell+1}^{a_{\ell+1}}\cdots p_k^{a_k}.
\end{align*}
By our comments above there must also be positive integers $b_1,\ldots, b_k$ for which
\begin{equation*}
d_{qr}=p_1^{b_1}\cdots p_k^{b_k}.
\end{equation*}
Now writing
\begin{align*}
q_dd_{qr}&=p_1^{\alpha_1}\cdots p_k^{\alpha_k},\quad\text{ and}\\
r_dd_{qr}&=p_1^{\beta_1}\cdots p_k^{\beta_k}
\end{align*}
we find that $\alpha_i,\beta_i\in\Z^+$ and that $\alpha_i\not=\beta_i$ for each $1\le i\le k$. By Lemma \ref{ZGaddlem2} this implies that
\begin{equation*}
F_{p_i^{\alpha_i}}\times F_{p_i^{\beta_i}}=
\begin{cases}
\varphi (p_i^{\alpha_i})F_{p_i^{\beta_i}}&\text{ if }1\le i\le\ell,\quad\text{ and}\\
\varphi (p_i^{\beta_i})F_{p_i^{\alpha_i}}&\text{ if }\ell <i\le k.
\end{cases}
\end{equation*}
Applying Lemma \ref{ZGaddlem1} then yields
\begin{align*}
F_{q_dd_{qr}}\times F_{r_dd_{qr}}&=\prod_{i=1}^k\left(\varphi \left(p_i^{\min\{\alpha_i,\beta_i\}}\right)F_{p_i^{\max\{\alpha_i,\beta_i\}}}\right)\\
&=\varphi (d_{qr})F_{q_dr_dd_{qr}}.
\end{align*}
Since \[\varphi (p^\alpha)-p^{\alpha-1}=\varphi (p^{\alpha})\frac{p-2}{p-1},\]
another application of Lemmas \ref{ZGaddlem1} and \ref{ZGaddlem2} gives us
\begin{align*}
F_{d'}\times F_{d'}&=\prod_{p^\alpha\|d'}\left(\varphi (p^{\alpha})\sum_{i=0}^{\alpha}c(d',d'/p^i)F_{p^i}\right)\\
&=\varphi (d')\sum_{e|d'}c(d',d'/e)F_e\\
&=\varphi (d')\sum_{e|d'}c(d',e)F_{d'/e}.
\end{align*}
Note that here we are also using the fact that $c(d',d'/e)$ is a multiplicative function of $e$. Returning to equation (\ref{ZGaddthmeqn1}) we now have that
\begin{align*}
F_q\times F_r&=F_{q_1r_1}\times \varphi (d_{qr})F_{q_dr_dd_{qr}}\times \varphi (d')\sum_{e|d'}c(d',e)F_{d'/e}\\
&=\varphi (d)\sum_{e|d'}c(d',e)F_{q_1r_1q_dr_dd_{qr}d'/e}\\
&=\varphi (d)\sum_{e|d'}c(d',e)F_{qr/de},
\end{align*}
which finishes the proof of Theorem \ref{ZGaddthm}.
\end{proof}
The assertions of Corollaries \ref{ZGaddcor1} and \ref{ZGaddcor2} follow readily from the Theorem. We leave the proofs to the reader.

%%%%%%%%%%%%%%%%%%%%%%%%%%%%%%%%%%%%%%%%%%%%%%%%%%%%%%%%%%%%%%%%%%%%
%%%%%%%%%%%%%%%%%%%%%%%%%%%%%%%%%%%%%%%%%%%%%%%%%%%%%%%%%%%%%%%%%%%%

\section{The cardinality of $\F_Q+\F_Q$}
In this section we will prove Theorem \ref{FQordthm}. We begin by writing
\[I_Q=|\F_Q+\F_Q|,\]
and by defining an arithmetical function
\begin{align*}
\tau^*_Q(n)=\sum_{\substack{ d \mid n \\ d,n/d \leq Q \\
(d,n/d)=1}} 1.
\end{align*}
By appealing to Corollary \ref{ZGaddcor2} we have that
\begin{align*}
I_Q=\sum_{\substack{n \leq Q^2 \\ \tau^*_Q(n) \geq 1}}
\varphi(n).
\end{align*}
A trivial upper bound for this quantity is given by
\begin{align*}
I_Q\le Q^2\cdot|\{n\le Q^2 : \tau^*_Q(n) \geq 1\}|.
\end{align*}
Estimating the cardinality of the set of integers which appears here is closely related to the ``multiplication table problem'' posed by Erd\"{o}s (\cite{erdos1955},\cite{erdos1960}). This problem is solved in \cite{ford2006}, where it is proved that the number of integers less than or equal to $Q^2$ which can be written as a product of
two positive integers $n,m\le Q$ is
\[O\left(\frac{Q^2}{(\log Q)^\delta (\log\log Q)^{3/2}}\right),\]
where $\delta=1-(1+\log \log 2)/\log 2$. This immediately gives us the upper bound in Theorem
\ref{FQordthm}. The lower bound is much more delicate and to deal
with it we will closely follow ideas developed by K. Ford. For
each $a\in\N$ we define $\mc{L}(a)\subset\R$ by
\[\mc{L}(a)= \bigcup_{d | a} [\log d - \log 2,\log d).\]
Roughly speaking, the Lebesgue measure of $\mc{L}(a)$ measures the clustering of the divisors of $a$. Our proof will rest crucially on the following estimate.
\begin{lemma}With $\mc{L}(a)$ defined as above we have as $Q\rar\infty$ that
\begin{align}\label{Fordest1}
\sum_{\substack{a \leq Q
\\ \mu(a) \ne 0}} \frac{\varphi (a)|\mc{L}(a)|}{a^2}\gg\frac{(\log Q)^{2-\delta}}{(\log\log Q)^{3/2}},
\end{align}
where $\delta$ is the same as in Theorem \ref{FQordthm} and $|\mc{L}(a)|$ denotes Lebesgue measure.
\end{lemma}
A similar bound, without the factor of $\varphi (a)/a$ in the summand, is proved in Section 2 of \cite{ford2006}. The proof there uses a clever application of the cycle lemma from combinatorics. The presence of the factor $\varphi (a)/a$ in our sum adds only minor difficulties. In fact all changes which arise by this modification are overcome by the multiplicativity of the function $\varphi (a)/a$, together with the fact that the sequences of integers under consideration are well distributed. Instead of giving a lengthy proof of our own, we leave this as a matter of fact for the reader to check.

Now if $n\in\N$ and $y,z\in\R$ then let $\tau(n;y,z)$ denote the
number of divisors of $n$ lying in the interval $(y,z]$. We will
obtain our lower bound by considering only special types of
integers in the sum $I_Q$. First note that
\begin{align}\label{IQbound3}
I_Q \geq
 \sum_{\substack{ n \leq \frac{Q^2}{2} \\ \tau(n;Q/2,Q) \geq 1 \\ \mu(n) \ne 0}} \varphi(n),
\end{align}
since if $d|n,~d> Q/2,~$ and $n$ is square-free, then we must
have that $n/d< Q$ and $(d,n/d)=1.$ Next write $y=y(Q)=Q/2$ and $x=x(Q)=Q^2/2$, and define $\A_Q$ to be the set of square-free integers $n\le x$ which can be factored in the form $n=apq$ with $1\le a \leq y^{1/8}$, $p$ a prime such that $\log (y/p) \in \mc{L}(a)$,
and $q > y^{1/8}$ is a prime. For reference we display this definition as
\begin{align*}
\A_Q=\left\{n=apq\le x :\mu (n)=0, a\le y^{1/8}<q, \log (y/p)\in\mc{L} (a)\right\}.
\end{align*}
If $n=apq\in\A_Q$ then
working from the definition of $\mc{L}(a)$ we see that $p$ lies
in an interval of the form $(y/d,2y/d]$ for some $d|a$. Thus we
have that $y^{7/8} \leq p \leq 2y$, and this shows that there can
be at most two representations of $n$ of the form which qualify it as an element of $\A_Q$ (another one
could possibly come from interchanging the roles of $p$ and $q$).
Furthermore we have that $dp\in (y,2y]$, so $\tau(n; y, 2y) \geq
1$ and comparing with (\ref{IQbound3}) we find that
\begin{align}\label{IQbound1}
I_Q\ge \sum_{n\in\A_Q}\varphi (n)\ge \frac{1}{2}\sum_{\substack{a \leq y^{\frac{1}{8}} \\ \mu(a) \ne
0}} \varphi(a) \sum_{ \log (y/p) \in \mc{L}(a)} (p-1)I_Q'(a,p),
\end{align}
where
\begin{align*}
I_Q'(a,p) &= \sum_{\substack{y^{1/8} < q \leq x/ap \\ (q,p)=1\\
q \text{ prime}}} (q-1).
\end{align*}
For a lower estimate of $I_Q'$ we employ the Prime Number Theorem
to obtain
\begin{align*}
I_Q'(a,p) \gg \frac{\left( \frac{x}{ap} \right)^2}{\log
\frac{x}{ap}} - \frac{y^{1/4}}{(\log y^{1/8})}.
\end{align*}
Since $y^{7/8}\le x/ap\le 2y^{9/8}$ we have that $\log
\frac{x}{ap}\ll \log y$ and that
\begin{align*}
I_Q'(a,p)\gg \frac{x^2}{a^2 p^2 \log y}.
\end{align*}
Substituting this back into (\ref{IQbound1}) gives us
\begin{align}\label{IQbound2}
I_Q &\gg \frac{x^2}{\log y} \sum_{\substack{a \leq
y^{\frac{1}{8}} \\ \mu(a) \ne 0}} \frac{\varphi (a)}{a^2}
 \sum_{ \log (y/p) \in \mc{L}(a)} \frac{p-1}{p^2}.
\end{align}
Estimating the inner sum here is not difficult. First we divide the
set $\mc{L}(a)$ into connected components. Each of these
components has the form $[\log d-\log c,\log d)$ for some $d|a$
and $c\ge 2$. Thus we have that
\begin{align*}
\sum_{\log (y/p)\in[\log (d/c),\log d)}\frac{1}{p}&=\sum_{p\in(y/d,yc/d]}\frac{1}{p}\\
&=\log(\log (yc/d))-\log(\log (y/d))+O\left(\exp (-c_1\sqrt{\log
x})\right),
\end{align*}
where $c_1>0$ is a positive constant coming from the error term
in the Prime Number Theorem. With a little manipulation it is not
difficult to see that the latter expression is
\begin{align*}
\frac{\log c}{\log (y/d)}+O\left(\frac{1}{(\log Q)^2}\right)\gg
\frac{\log c}{\log Q}.
\end{align*}
Returning to (\ref{IQbound2}) we now have that
\begin{align*}
I_Q\gg \frac{Q^4}{(\log Q)^2} \sum_{\substack{a \leq
y^{\frac{1}{8}}
\\ \mu(a) \ne 0}} \frac{\varphi (a) |\mc{L}(a)|}{a^2},
\end{align*}
and combining this with (\ref{Fordest1}) finishes the proof of
our lower bound and of Theorem \ref{FQordthm}.

As we mentioned in the Introduction, it is an open problem to determine the order of magnitude of the cardinality of (\ref{FQsumset1}) when $k\ge 3$. The techniques used here would require some serious modifications to deal with those cases. Let us write $I_Q(k)$ for the cardinality of the sets in question. It follows from equation (\ref{GQdef1}) and Corollary \ref{ZGaddcor2} that for $k\ge \pi (Q)$ we have that
\[I_Q(k)=|G_Q|.\] We leave the reader with the following question. How small may we take $k$ (as a function of $Q$) and still conclude that
\[\log I_Q(k)\asymp\log |G_Q|?\]
This seems to be an interesting problem which may possibly be solved by techniques that differ from those used here.

\end{document}